\documentclass[a4paper]{amsart}
\usepackage{dsfont}
\usepackage{amsmath,amsthm,amssymb,hyperref,mathrsfs}
\usepackage{aliascnt}
\usepackage{lmodern}
\usepackage[T1]{fontenc}
\usepackage[textsize=footnotesize]{todonotes}

\usepackage{xcolor}
\definecolor{dblue}{rgb}{0,0,0.70}
\hypersetup{
	unicode=true,
	colorlinks=true,
	citecolor=dblue,
	linkcolor=dblue,
	anchorcolor=dblue
}

\makeatletter
\expandafter\g@addto@macro\csname th@plain\endcsname{%
		\thm@notefont{\bfseries}
	}%
\expandafter\g@addto@macro\csname th@remark\endcsname{%
		\thm@headfont{\bfseries}
	}%
\makeatother

\newtheorem{theorem}
{Theorem}[section]
\newtheorem*{theorem*}{Theorem}

\newaliascnt{lemma}{theorem}
\newtheorem{lemma}[lemma]{Lemma}
\aliascntresetthe{lemma}
\newtheorem*{lemma*}{Lemma}

\newaliascnt{fact}{theorem}

\aliascntresetthe{fact}

\newaliascnt{proposition}{theorem}
\newtheorem{proposition}[proposition]{Proposition}
\aliascntresetthe{proposition}
\newtheorem*{proposition*}{Proposition}

\newaliascnt{corollary}{theorem}
\newtheorem{corollary}[corollary]{Corollary}
\aliascntresetthe{corollary}

\theoremstyle{remark}

\newaliascnt{remark}{theorem}
\newtheorem{remark}[remark]{Remark}
\aliascntresetthe{remark}
\newaliascnt{question}{theorem}

\aliascntresetthe{question}
\newaliascnt{conjecture}{theorem}
\newtheorem{conjecture}[conjecture]{Conjecture}
\aliascntresetthe{conjecture}

\newtheorem*{question*}{Question}

\newaliascnt{definition}{theorem}
\newtheorem{definition}[definition]{Definition}
\aliascntresetthe{definition}

\newaliascnt{example}{theorem}

\aliascntresetthe{example}

\newcommand{\axiom}[1]{\mathsf{#1}}
\newcommand{\ZFC}{\axiom{ZFC}}
\newcommand{\AC}{\axiom{AC}}

\newcommand{\ZF}{\axiom{ZF}}

\newcommand{\Ord}{\mathrm{Ord}}

\newcommand{\KWP}{\axiom{KWP}}
\newcommand{\SVC}{\axiom{SVC}}

\DeclareMathOperator{\tcl}{tcl}

\DeclareMathOperator{\Def}{Def}

\newcommand{\forces}{\mathrel{\Vdash}}

\newcommand{\power}{\mathcal{P}}

\newcommand{\PP}{\mathbb P}
\newcommand{\QQ}{\mathbb Q}

\newcommand{\tup}[1]{\langle#1\rangle}

\author{Asaf Karagila}
\author{Jonathan Schilhan}
\email{karagila@math.huji.ac.il}
\urladdr{https://karagila.org}
\email{jonathan.schilhan@univie.ac.at}
\urladdr{https://www.logic.univie.ac.at/~jschilhan/}
\address{School of Mathematics,
    University of Leeds.
    Leeds, LS2~9JT, UK}

\date{10 July, 2025}

\title{Intermediate models and Kinna--Wagner Principles}

\begin{document}
\begin{abstract}
  Kinna--Wagner Principles state that every set can be mapped into some fixed iterated power set of an ordinal, and we write $\KWP$ to denote that there is some $\alpha$ for which this holds. The Kinna--Wagner Conjecture, formulated by the first author in \cite{Karagila:2020}, states that if $V$ is a model of $\ZF+\KWP$ and $G$ is a $V$-generic filter, then whenever $W$ is an intermediate model of $\ZF$, that is $V\subseteq W\subseteq V[G]$, then $W=V(x)$ for some $x$ if and only if $W$ satisfies $\KWP$. In this work we prove the conjecture and generalise it even further. We include a brief historical overview of Kinna--Wagner Principles and new results about Kinna--Wagner Principles in the multiverse of sets.
\end{abstract}
 \maketitle              
\section{Introduction}
The Axiom of Choice was formulated by Zermelo in order to prove that every set can be well-ordered, and we know that the converse statement holds as well. Namely, if we assume that every set can be well-ordered, then the Axiom of Choice holds. Since every well-ordered set is isomorphic to a unique ordinal, we can phrase the Axiom of Choice as ``Every set injects into an ordinal''.

Kinna and Wagner formulated a weakening of the Axiom of Choice in \cite{KWoriginal}. Formulated as a selection principle, they prove that it is equivalent to ``Every set injects into the power set of an ordinal''. So, an immediate consequence of this principle is that every set can be linearly ordered. The Kinna--Wagner selection principle was studied extensively, both by looking at various weaker versions (mostly derived from the formulation of the selection principle, see \cite{HRR:KW} for example), as well as studying its independence and consequences in the broader context of choice principles which imply linear ordering, see \cite{Pincus:dense} for example).

In this paper we look at a different weakening of the Kinna--Wagner Principle, derived from the concept of ``Every set injects into the $n$th-power set of an ordinal''. This generalisation is originally due to Monro \cite{Monro:KW}, later extended to the transfinite hierarchy by the first author in \cite{Karagila:Bristol,Karagila:Sym}. Kinna--Wagner Principles are fairly stable under generic extensions, but can be violated using symmetric extensions to a certain degree. The \emph{Kinna--Wagner Conjecture} states, informally, that if $V$ satisfied any Kinna--Wagner Principle, then the intermediate models between $V$ and $V[G]$, a set-generic extension of $V$, which satisfy some kind of Kinna--Wagner Principle are exactly of the form $V(x)$.

The Kinna--Wagner Conjecture can be understood as a generalisation of the intermediate model theorem, which states that if $V\subseteq M\subseteq V[G]$ are models of $\ZFC$, where $G$ is a $V$-generic filter, then $M$ is a generic extension of $V$ and $V[G]$ is a generic extension of $M$. This is not true if $M$ does not satisfy the Axiom of Choice, as generic extensions preserve the Axiom of Choice. But a natural question is whether $M$ must be a symmetric extension of $V$, at least in the presence of the Axiom of Choice, or some Kinna--Wagner Principle.

Grigorieff \cite{Grigorieff:Intermediate} studied intermediate models of $\ZF$, mainly in the form of symmetric extensions, even without the assumption of $\AC$ in $V$ or $V[G]$. The work on the Bristol model in \cite{Karagila:Bristol,Karagila:2020} showed that if $c$ is a Cohen real over $L$, then the structure of intermediate models is far more complex than previously imagined. Not every intermediate model has the form $L(x)$ for any set $x$, and even those that do, are not necessarily symmetric extension given by the Cohen real itself. These results were recently extended by Hayut and Shani \cite{HayutShani}, where they show that given any model of $\ZF$, $V$, if $c$ is a Cohen real over $V$, then there is an intermediate $V\subseteq M\subseteq V[c]$ such that $M\neq V(x)$ for all $x$. Hayut and Shani show that in these intermediate models all Kinna--Wagner Principles fail.

We show in the final section of this work that more can be said. \autoref{thm:KWC} shows that if $V\subseteq M\subseteq V[G]$ are models of $\ZF$, with $G$ a $V$-generic filter, and some Kinna--Wagner Principle holds in $V$, then $M$ has the form $V(x)$ for some set $x$ if and only if it satisfies some Kinna--Wagner Principle, and that the set $x$ is uniformly defined from the Kinna--Wagner Principle. The notion of a Kinna--Wagner Principle is generalised to allow different ``starting classes'' than the ordinals, and consequently we obtain \autoref{thm:general-form}.

\section{Preliminaries}
We will use a superscript with $M$ to denote the relativisation of classes to $M$, e.g.\ $\power(x)^M$ would be the power set of $x$ as computed in $M$. The exception would be the von Neumann hierarchy of $M$, which we denote by $M_\alpha$ rather than $V_\alpha^M$. All our ``models of $\ZF$'' are class models, they are all transitive classes in the context of any larger model. We will also write $\power(X)$, even if $X$ is a proper class, to mean $\{x\mid x\subseteq X\}$, but in the case where $X$ is a transitive class\footnote{Which will be the typical case.} this is tantamount to $\bigcup\{\power(x)\mid x\in X\}$.

Suppose that $x$ is a set, we define $L(x)$ to be the smallest transitive model of $\ZF$ containing the ordinals and $x$ itself. This model has a constructible hierarchy, much like $L$ does, with the modification that $L_0(x)=\tcl(\{x\})$, rather that $\varnothing$. We write $L(x,y)$ to mean $L(\{x,y\})$.

We can further extend this notation and define for an inner model $V$ and a set $x$, the class $V(x)$ to be the smallest transitive model of $\ZF$ containing $V$ and having $x$ as an element. The following equalities hold: \[V(x)=\bigcup\{L(x,y)\mid y\in V\}=\bigcup\{L(V_\alpha,x)\mid\alpha\in\Ord\}=\bigcup\{L_\alpha(V_\alpha,x)\mid\alpha\in\Ord\}.\]

Often, especially in the context of $\ZFC$, we take $V$ to be some definable inner model, however, this can be defined in any second-order set theory such as G\"odel--Bernays set theory, for any inner model, or by adding $V$ as a predicate when possible (e.g., when $V$ is a ground of the universe).

The definition of $V(x)$ can be also extended to proper classes, although this requires a bit more care. Given $V\subseteq W$ and $X\subseteq W$, we define \[V(X)=\bigcup\{V(X\cap W_\alpha)\mid\alpha\in\Ord\}.\] This definition makes the most sense when $W$ is a model of a second-order set theory and both $V$ and $X$ are proper classes of $W$, but this can be easily understood in general.

Finally, we will use the shorthand notation $x\leq y$ to mean that there exists an injection from the set $x$ into the set $y$. We use the $x\leq^*y$ notation to denote that either $x$ is empty or that there is a surjection from $y$ onto $x$, or alternatively, that there is a partial function from $y$ onto $x$.

\subsection{Generic and symmetric extensions}
\emph{Generic extensions} are given by forcing. We will hardly use the mechanics of forcing and we refer the reader to standard sources, e.g.\ Jech \cite{Jech:ST}, for a good review of the technique. While generic extensions are incredibly useful in the study of $\ZFC$, they preserve the Axiom of Choice from the ground model, and as such are not quite the needed tool for proving independence results related to the Axiom of Choice.

\emph{Symmetric extensions} are formed by identifying an intermediate model between $V$ and $V[G]$. We will not be using the technique directly, but we refer the readers to \cite{KS:Dist} for an overview of the technique, along with several interesting\footnote{In the humble view of the authors, at least.} results on the technique, as well as the limitations of forcing in $\ZF$.

One of the most important choice principles related to symmetric extensions was given by Blass \cite{Blass:1979}, where he suggests that the failure of choice is, in a way, due to ``one bad apple'', or ``small''.

\begin{definition}[Small Violations of Choice]
  We say that $V\models\SVC(X)$ if for every set $A$ there is an ordinal $\eta$ and a surjection $f\colon X\times\eta\to A$. We write $V\models\SVC$ to mean that there exists some $X$ such that $\SVC(X)$ holds.
\end{definition}
We can define an injective version of $\SVC$, namely, requiring that $A$ injects into $X\times\eta$. However, the two notions are equivalent when working in $\ZF$, as the surjection from $X\times\eta$ onto $A$ translates to an injection from $A$ into $\power(X\times\eta)$, which can be then modified to an injection into $\power(A)\times\eta$. For a brief discussion on this subject, see \cite{Ryan-Smith:2024}.

Small Violation of Choice turned out to be an incredibly useful choice principle. Through a combination of results due to Blass \cite{Blass:1979} and Usuba \cite{Usuba:2021} the following theorem holds.
\begin{theorem}\label{thm:BlassUsuba}
  The following are equivalent for $W\models\ZF$.
  \begin{enumerate}
  \item $W\models\SVC$.
  \item The Axiom of Choice can be forced over $W$.
  \item $W$ is a symmetric extension of a model of $\ZFC$.
  \item $W=V(x)$ where $V\models\ZFC$ and $x$ is a set.\qed
  \end{enumerate}
\end{theorem}
\section{Kinna--Wagner Principles}
\subsection{Historical background}
The Axiom of Choice can be phrased as ``For every set $M$, there is a function $f\colon\power(M)\to\power(M)$ such that $f(A)\subseteq A$ for all $A\in\power(M)$, and if $A$ is non-empty, then $f(A)$ is a singleton''. This lends itself to a natural weakening. Instead of selecting a singleton, we merely select a non-empty proper subset when possible. This is the original formulation of the Kinna--Wagner selection principle in \cite{KWoriginal}.

Formally, the \emph{Kinna--Wagner Selection Principle} states that for every set $M$ there is a function $f\colon\power(M)\to\power(M)$ such that $f(A)\subseteq A$ for all $A\in\power(M)$, and if $|A|\geq 2$, then $\varnothing\subsetneq f(A)\subsetneq A$.

In their original paper, Kinna and Wagner prove that $\power(M)$ admits such a function if and only if $M$ itself can be injected into the power set of some well-ordered set. Let us sketch a brief proof of this equivalence.

If $M\subseteq\power(\eta)$ and $|A|\geq 2$, we define $\alpha_A=\min(\bigcup A\setminus\bigcap A)$, note that if $|A|\geq 2$ then this is well-defined. Next define $f(A)=\{a\in A\mid\alpha_A\notin a\}$, or $f(A)=\varnothing$ if $|A|\leq 1$. Then $f(A)$ produces a non-empty, proper subset of $A$ as wanted.

In the other direction, starting from $f\colon\power(M)\to\power(M)$ we can build a tree by recursion, starting from the root as $M$ and at each node, $A$, we have two successors: $f(A)$ and $A\setminus f(A)$. At limit points we take intersections along the branches (this is a copy of $2^{<\alpha}$ for some $\alpha$, so branches do exist), and if they have more than two points, we continue. This recursion must end at some stage $\eta$, which can only happen once all the terminal nodes are singletons and every singleton has been reached. For every $m\in M$ we consider the ``trace'' of $m$ in this tree: the set of all $\alpha<\eta$ such that for some $A$ in the $\alpha$th level of the tree, $m\in f(A)$. Easily this provides us an injection from $M$ into $\power(\eta)$.

Many selection principles were derived from the one above, often by restricting the size of $M$ or requiring that the size of the selected set will have some properties (e.g., \cite{HRR:KW}). Monro, however, generalised these principles based on the embeddibility into iterated power sets, starting from an ordinal in \cite{Monro:KW}, extended further by the first author in \cite{Karagila:Sym}, as well as studied by Shani in \cite{Shani:2021}.

\subsection{Higher-order generalisation}
We define the iterated power set of a set $x$, $\power^\alpha(x)$, by recursion: \begin{enumerate}
\item $\power^0(x)=x$,
\item $\power^{\alpha+1}(x)=\power(\power^\alpha(x))$, and
\item $\power^\alpha(x)=\bigcup\{\power^\beta(x)\mid\beta<\alpha\}$ when $\alpha$ is a limit ordinal.
\end{enumerate}
In the case where $x$ is an infinite set and $x^2\leq x$, we immediately get that $(\power^\alpha(x))^2\leq\power^\alpha(x)$ for all $\alpha$. In the case where $x$ is an infinite ordinal (or $\Ord$ itself) this allows us to encode any binary relation on $\power^\alpha(\Ord)$ as a subset of $\power^\alpha(\Ord)$.
\begin{definition}
  The \textbf{Kinna--Wagner $\boldsymbol\alpha$} principle, denoted by $\KWP_\alpha$, states that every set $x$ injects into $\power^\alpha(\Ord)$. That is, there is an ordinal $\eta$ such that $x$ injects into $\power^\alpha(\eta)$. We omit the index to write $\KWP$ as a shorthand for $\exists\alpha\, \KWP_\alpha$.

  We define also the surjective version, \textbf{Kinna--Wagner* $\boldsymbol\alpha$} principle, denoted by $\KWP_\alpha^*$, which asserts that for every set $x$, there is a partial function from $\power^\alpha(\Ord)$ onto $x$. Namely, there is some ordinal $\eta$ and a partial function from $\power^\alpha(\eta)$ onto $x$. We use $\KWP^*$ to denote $\exists\alpha\, \KWP_\alpha^*$ as well.
\end{definition}
The next proposition is an immediate corollary, following from the fact that for every $x$ and $y$, $x\leq y\to x\leq^* y\to x\leq\power(y)$.
\begin{proposition}
  For all $\alpha$, $\KWP_\alpha\to\KWP^*_\alpha\to\KWP_{\alpha+1}$.\qed
\end{proposition}
Consequently, $\KWP$ is equivalent to $\KWP^*$, but we can say even more.
\begin{proposition}
  If $\alpha$ is not a successor ordinal, then $\KWP_\alpha=\KWP_\alpha^*$.
\end{proposition}
In the case of $\alpha=0$ this is immediate, as both $\KWP^*_0$ and $\KWP_0$ are equivalent to the Axiom of Choice. In the case of limit ordinals this is a corollary of a much more general lemma.
\begin{lemma}
  Suppose that $\alpha$ is a limit ordinal, then $x\leq\power^\alpha(y)\iff x\leq^*\power^\alpha(y)$.
\end{lemma}
\begin{proof}
  Since $\ZF$ proves that if $x\leq\power^\alpha(y)$, then $x\leq^*\power^\alpha(y)$, we only need to verify the other implication. Suppose that $x\leq^*\power^\alpha(y)$. If $x$ is empty, then $x\subseteq\power^\alpha(y)$, so we may assume that it is not empty. Therefore, there is a surjection $F\colon\power^\alpha(y)\to x$. For each $u\in x$, let $\alpha_u=\min\{\beta\mid F^{-1}(u)\cap\power^\beta(y)\neq\varnothing\}$, this is a well-defined ordinal since $F$ is surjective and $\power^\alpha(y)=\bigcup\{\power^\beta(y)\mid\beta<\alpha\}$.

  Define $f\colon x\to\power^\alpha(y)$ given by $f(u)=\{v\in\power^{\alpha_u}(y)\mid F(v)=u\}$. Since $\alpha$ is a limit ordinal, $f(u)\in\power^{\alpha_u+1}(y)\subseteq\power^\alpha(y)$, so it is well-defined, and since $F$ is a function, $f$ must be injective.
\end{proof}
\begin{theorem}[Balcar--Vop\v{e}nka--Monro]\label{thm:BVM}
  Suppose that $M$ and $N$ are two models of $\ZF$ with $\power^{\alpha+1}(\Ord)^M=\power^{\alpha+1}(\Ord)^N$, if $M\models\KWP^*_\alpha$, then $M=N$.
\end{theorem}
\begin{proof}
  Since $M$ satisfies $\KWP^*_\alpha$, every set in $M$ is the image of $\power^\alpha(\eta)$ of some ordinal $\eta$. Therefore, in $M$ we can code $\tcl(\{x\})$ as the extensional quotient of a relation on $\power^\alpha(\eta)$ for some $\eta$, and by iterating G\"odel's pairing, we have some $X\in\power^{\alpha+1}(\eta)$ from which we can recover the extensional and well-founded $\in$-relation on $\tcl(\{x\})$. By the assumption, $X\in N$, and it is still a well-founded relation (otherwise there would be a subset of $X$ witnessing that in both $M$ and $N$), so we can take the extensional quotient and use the Mostowski collapse lemma in $N$ and get $x\in N$ as wanted, so $M\subseteq N$.

  Suppose that $N\neq M$, then let $\delta$ be the least ordinal such that $M_\delta\neq N_\delta$. Note that both models have the same ordinals, so it is certainly the case that $M_\delta$ exists. Easily, $\delta$ must be a successor ordinal of the form $\gamma+1$, so in $N$ there is some $x\subseteq M_\gamma$ such that $x\notin M$. However, since $M_\gamma=N_\gamma$ can be coded as an element of $\power^{\alpha+1}(\eta)$ for some ordinal $\eta$, and therefore $x$ can be coded as a subset of that, in $N$, but the assumption guarantee that this code is in $M$, and so $x\in M$ as well.
\end{proof}
\begin{corollary}\label{cor:kwp-iterated-pset}
  Suppose that $V\models\KWP_\alpha^*$, then $V=L(\power^{\alpha+1}(\Ord))$.\qed
\end{corollary}
\begin{proposition}\label{prop:rel-KWP}
  Suppose that $V$ satisfies $\KWP$ and $x$ is a set, then $V(x)\models\KWP$. Moreover, if $V\models\KWP_\beta^*$ and $x\in\power^{\alpha+1}(\Ord)$, then $V(x)\models\KWP_{\max\{\alpha,\beta\}}^*$.
  \end{proposition}
  \begin{proof}
    Let $\gamma = \max\{\alpha, \beta\}$. Recall that $V(x)=\bigcup\{L(x,y)\mid y\in V\}$, and more specifically, $V(x)=\bigcup\{L(V_\delta,x)\mid\delta\in\Ord\}$. So, if $z\in V(x)$, then there is some $\delta$ such that $z\in L(V_\delta,x)$. It is enough to show that $L(V_\delta,x)\models\KWP_\gamma^*$ for all $\delta$.

    Recall that in general $L(a)$ has a recursive construction given by $L_0(a)=\tcl(\{a\})$, $L_{\xi+1}(a)=\Def(L_\alpha(a),\in)$, where $\Def$ denotes all the definable (with parameters) subsets of a given structure, and $L_\xi(a)=\bigcup\{L_\zeta(a)\mid\zeta<\xi\}$ for a limit ordinal $\xi$.

    We will construct, in $V(x)$, a surjection from $\power^\gamma(\Ord)$ onto $L(V_\delta,x)$, for any fixed $\delta$, given by recursion, where we keep track of the surjections constructed in previous steps. Since $x\subseteq\power^\alpha(\Ord)$ and $V_\delta$ is the image of $\power^\beta(\eta)$ for some $\eta$, there is a surjection, $f_0$, from $\power^\gamma(\Ord)$ onto $L_0(V_\delta,x)$. If $\xi$ is a limit ordinal and $\tup{f_\zeta\mid\zeta<\xi}$ is known, then by shifting the domains of the $f_\zeta$---if necessary---we get an obvious surjection from $\power^\gamma(\Ord)$ onto $L_\xi(V_\delta,x)$. Finally, for a successor step, there is a surjection, the interpretation map, from $\omega\times L_\xi(V_\delta,x)^{<\omega}$ onto $L_{\xi+1}(V_\delta,x)$ given by mapping $\tup{n,\vec a}$ to the set defined by $\varphi_n(x,\vec a)$, where $\varphi_n$ is some enumeration of all formulas with the understanding that we map the pair to $\varnothing$ if there is an arity mismatch.\footnote{We can also agree that any unassigned parameters are given the value $0$, etc.} Since $\omega\times L_\xi(V_\delta,x)$ is the image of $\power^\gamma(\Ord)$, given by $f_\xi$, we can define $f_{\xi+1}$ as the composition of $f_\xi$ with the interpretation map.
  \end{proof}
\begin{remark}
  The converse statement, $V(x)\models\KWP\to V\models\KWP$ is false. If $M\subseteq L[c]$ is the Bristol model, then $M(c)=L[c]\models\KWP_0$, but $M\models\lnot\KWP$. In \autoref{thm:kwp-down-abs} we will show that the statement does hold for the case of generic extensions (and consequently, symmetric extensions).
\end{remark}
\begin{proposition}
  $L(\power^\alpha(\Ord))\models\KWP_\alpha^*$.
\end{proposition}
\begin{proof}
  Note that $L(\power^\alpha(\Ord))=\bigcup\{L(\power^\alpha(\delta))\mid\delta\in\Ord\}$. By the \autoref{prop:rel-KWP}, $L(\power^\alpha(\delta))\models\KWP_\alpha^*$ for all $\delta$, so for every $x\in L(\power^\alpha(\Ord))$ there is some $\delta$ such that $x\in L(\power^\alpha(\delta))$ and therefore there is some $\eta$ such that $\power^\alpha(\eta)$ maps onto $x$.
\end{proof}
The following proposition follows directly from the definition of $\SVC$ and the fact that $\power^\alpha(\Ord)$ definably maps onto $\power^\alpha(\Ord)\times\Ord$.
\begin{proposition}\label{prop:SVCtoKWP}
  If $\SVC(x)$ holds and $x\in\power^{\alpha+1}(\Ord)$, then $\KWP_\alpha^*$ holds. In particular $\SVC\to\KWP$.\qed
\end{proposition}
Of course, if the injective version of $\SVC$ holds, we get $\KWP_\alpha$. This is the core of the proof that in Cohen's first model every set can be linearly ordered: there is a set of reals, $A$, such that $\SVC([A]^{<\omega})$ holds in its injective form (see an example of such proof in \cite[\S5.5]{Jech:AC}). It should be noted that $\ZF$ does not prove that $\KWP\to\SVC$. Monro \cite{Monro:1975} constructed a model in which there is a proper class of Dedekind-finite sets using a class forcing, where $\SVC$ fails, but it can be shown that $\KWP_1$ holds in that model.
\begin{lemma}\label{lemma:a-set-gen}
 Suppose $M = V(x)$ for some $V\models\ZF$ and a set $x$ and $M \models \KWP^*_\alpha$. Then $M = V(y)$ for some $y \in \power^{\alpha+1}(\Ord)$.
\end{lemma}
\begin{proof}
 Since $M \models \KWP^*_\alpha$, there is some $\eta \in \Ord$ and a surjection from $\mathcal{P}^\alpha(\eta)$ to $\tcl(\{x \})$. Exactly as in the proof of \autoref{thm:BVM}, we find $y \subseteq \mathcal{P}^\alpha(\eta)$ in $M$ from which we can recover $x$.
\end{proof}
\section{The Kinna--Wagner Conjecture and its proof}
\subsection{Kinna--Wagner in the forcing multiverses}
\begin{proposition}\label{prop:kwp-up-abs}
  If $V\models\KWP_\alpha^*$ and $G$ is a $V$-generic filter, then $V[G]\models\KWP_\alpha^*$.
\end{proposition}
\begin{proof}
  Let $x\in V[G]$ be any non-empty set, then there is some $\dot x\in V$ such that $\dot x^G=x$. Fixing some $x_0\in x$, let $F\colon\dot x\to x$ be the function $F(p,\dot y)=\dot y^G$ when $p\in G$, otherwise $F(p,\dot y)=x_0$. This function is surjective, since $\dot x\in V$, there is an ordinal $\eta$ such that $\dot x$ injects into $\power^\alpha(\eta)$, and so $F$ can be extended to $\power^\alpha(\eta)$ via precomposition.
\end{proof}
As an immediate corollary, if $V\models\KWP_\alpha$, then $V[G]\models\KWP_\alpha^*$ and therefore $\KWP_{\alpha+1}$, for any generic filter $G$. On the other hand, Monro showed that there is a generic extension of Cohen's first model in which there is an amorphous set \cite{Monro:Generic}. Since Cohen's first model satisfies $\KWP_1$, and amorphous sets cannot be linearly ordered, it must be that Monro's generic extension satisfies $\KWP_1^*+\lnot\KWP_1$. However, by the above proposition, any further extension must satisfy $\KWP_1^*$, since generic extensions of generic extensions are themselves generic extensions.

In the modal logic of forcing, $\KWP_\alpha^*$ is a button. Once it is true, it must remain true in any further generic extension (it may be true that for some $\beta<\alpha$, $\KWP_\beta^*$ holds in a generic extension, of course). The converse statement is not necessarily true, of course, since if $c$ is a Cohen real over $L$, then $L[c]\models\KWP^*_0$, but it has grounds where $\KWP_\alpha^*$ fails for arbitrarily high $\alpha$, as shown in \cite{Karagila:2020}. On the other hand, as the following theorem shows, the failure cannot be complete.

\begin{theorem}\label{thm:kwp-down-abs}
  If $V[G]\models\KWP_\alpha^*$, where $G\subseteq\PP\subseteq\power^\beta(\Ord)$ is $V$-generic, then $V\models\KWP_{\beta+\alpha}^*$.
\end{theorem}
\begin{proof}
  We define a sequence of classes of $\PP$-names, $S_\delta$, by recursion:
  \begin{enumerate}
  \item $S_0=\{\check\xi\mid\xi\in\Ord\}$,
  \item $S_{\delta+1}=\power(\PP\times S_\delta)$,
  \item $S_\delta=\bigcup_{\gamma<\delta}S_\gamma$ for a limit ordinal $\delta$.
  \end{enumerate}

  If $a\in\power^\delta(\Ord)^{V[G]}$, then there is some $\dot a\in S_\delta$ such that $\dot a^G=a$. Moreover, in $V$ there is a definable surjection $\tau_\delta\colon\power^{\beta+\delta}(\Ord)\to S_\delta$, since $\PP\subseteq\power^\beta(\Ord)$.

  Now letting $x\in V$ be an arbitrary set, there are $\PP$-names $\dot a\in S_{\alpha+1}$ and $\dot f$, and a condition $p\in G$ such that $p\forces``\dot f\colon\dot a\to\check x$ is a surjection''. Let $g$ be the function $g(q,s)=y$ if and only if $q\forces\dot f(\tau_\alpha(s))=\check y$, then $g$ is a surjection from a subset of $\PP\times\power^{\beta+\alpha}(\Ord)$ onto $x$. So $V\models\KWP_{\beta+\alpha}^*$, as wanted.
\end{proof}
\begin{corollary}
  Suppose that $V\models\lnot\KWP_{\beta+\alpha}$ and $G\subseteq\PP\subseteq\power^\beta(\Ord)$ is a $V$-generic filter, then $V[G]\models\lnot\KWP_\alpha$. In particular, you cannot force the Axiom of Choice with a well-orderable forcing from a model of $\ZF+\lnot\AC$.\footnote{Note that the latter is also a consequence of the fact that a well-orderable forcing must preserve empty products. That is, if $\prod\{A_i\mid i\in I\}=\varnothing$, then adding a generic for a well-orderable forcing must preserve that.}
\end{corollary}

This means that $\KWP$ is an absolute truth not only in the generic multiverse, but indeed in the symmetric multiverse. Grigorieff showed in \cite{Grigorieff:Intermediate} that if $V\subseteq W\subseteq V[G]$, where $W$ is a symmetric extension of $V$, then $V[G]$ is a generic extension of $W$. Therefore, if $V\models\KWP$, so must $V[G]$, and therefore $W$ must satisfy this as well. The Bristol model is an example of an intermediate model between $L$ and $L[c]$, where $c$ is a Cohen real, where $\KWP$ fails. Of course, the Bristol model is not of the form $L(x)$ for any set $x$, which leads to the very natural conjecture made by the first author in \cite[Conjecture~8.9]{Karagila:2020}.

\begin{conjecture}[The Kinna--Wagner Conjecture]
  Suppose that $V\models\KWP$ and $G$ is a $V$-generic filter. If $M$ is an intermediate model between $V$ and $V[G]$ and $M\models\KWP$, then $M=V(x)$ for some set $x$.
\end{conjecture}
\subsection{Proof of the Kinna--Wagner Conjecture}
We begin with a lemma, before proving the Kinna--Wagner Conjecture.
\begin{lemma}\label{lemma:cBA-subset}
  Let $W\models\ZF$ and let $\QQ\in W$ be a forcing notion, $f\colon X\to\QQ$ a function whose image is dense, and $H\subseteq\QQ$ is a $W$-generic filter. For every $x\in W[H]$ such that $x\subseteq W$, there is $y\subseteq X$ such that $W(x)=W(y)$.
\end{lemma}
\begin{proof}
  It is enough to prove this in the case where $\QQ$ is a complete Boolean algebra, since every forcing notion embeds densely into its Boolean completion. Let $\dot x$ be a $\QQ$-name such that $\dot x^H=x$, and let $\QQ_{\dot x}$ be the complete subalgebra generated by the conditions of the form $[\![\check z\in\dot x]\!]$ for $z\in V$. Then $W[H\cap\QQ_{\dot x}]=W[x]$, see the proof of \cite[Corollary~15.42]{Jech:ST}.\footnote{The proof in Jech is in the context of $\ZFC$, but the Axiom of Choice is not used in the proof beyond the fact that $\KWP_0$ and the Balcar--Vop\v{e}nka theorem reduce all cases to sets of ordinals and that sets of ordinals are subsets of $V$.}

  Let $\pi\colon\QQ\to\QQ_{\dot x}$ be the projection map given by $\pi(q)=\inf\{p\in\QQ_{\dot x}\mid q\leq p\}$. This is well-defined since $\QQ_{\dot x}$ is a complete subalgebra, so the computation of the infimum is the same in both $\QQ$ and $\QQ_{\dot x}$. Since $\pi``f``X$ is dense in $\QQ_{\dot x}$, by \cite[Lemma~15.40]{Jech:ST}, \[W[x]=W[H\cap\QQ_{\dot x}]=W[H\cap\pi``f``X],\] and therefore taking $y=f^{-1}(\pi^{-1}(H\cap\QQ_{\dot x}))\subseteq X$ and we get that $W[x]=W[y]$.
\end{proof}
One can view this lemma from a different direction: we use $f$ to define a quasi-order on $X$ which forcing equivalent to $\QQ$, and then apply the restriction lemmas.

We can now proceed with our proof of the Conjecture. We will, in fact, prove a stronger form of it.

\begin{theorem}\label{thm:KWC}
  There is a definable sequence $\tup{\eta_\alpha\mid\alpha\in\Ord}$ in $V[G]$, using $V$ as a predicate, such that given any intermediate model $V\subseteq M\subseteq V[G]$ of $\KWP_\alpha^*$, then $M=V(x)$, for some $x\subseteq\power^\alpha(\eta_\alpha)$.
\end{theorem}
As the definability of the ground model is still an open problem in $\ZF$ at the time of writing this paper, it is not even clear if $V$ itself is definable, which is why we must add it as a predicate. Note that even in this context, we do not require (a priori) that $M$ is definable in $V[G]$ (even with $V$ as a predicate), although that is a consequence of the theorem.

\begin{proof}[Proof of \autoref{thm:KWC}]
  We begin by defining a different sequence, $\tup{\eta'_\alpha\mid\alpha\in\Ord}$ in $V[G]$, using $V$ as a predicate. Define $\eta'_0=0$, for a limit ordinal $\alpha$ we let $\eta'_\alpha=\sup\{\eta'_\beta\mid\beta<\alpha\}$. Suppose the $\eta'_\alpha$ was defined, let us define $\eta_{\alpha+1}'$.

First let $\beta$ be the least such that $V\models\KWP_\beta^*$. Next, for each $z\subseteq\power^{\beta+\alpha}(\eta_\alpha')$, there is a complete Boolean algebra $\QQ\in V(z)$ and a $V(z)$-generic filter, $H$, such that $V(z)[H]=V[G]$. By \autoref{prop:rel-KWP}, $V(z)\models\KWP_{\beta+\alpha}^*$, and therefore there is an ordinal $\eta$ and a surjection from $\power^{\beta+\alpha}(\eta)^{V(z)}$ onto $\QQ$ in $V(z)$, we let $\eta^z$ be the least such ordinal. Then, let $\eta'_{\alpha+1}=\sup\{\eta^z\mid z\subseteq\power^{\beta+\alpha}(\eta'_\alpha)\}$.

Let $M$ be an intermediate model of $\KWP^*_\alpha$. We claim that $M=V(z)$ for some $z\subseteq\power^{\beta+\alpha}(\eta'_{\alpha+1})^M$. Let $X_\delta$ denote $\power^\delta(\Ord)^M$. First, note that $V(X_\delta)\models\KWP_{\beta+\delta}$, and since $M\models\KWP^*_\alpha$, $M=L(X_{\alpha+1})=V(X_{\alpha+1})$. We claim that $V(X_\delta)=V(z)$ for some $z\subseteq\power^{\beta+\delta}(\eta'_\delta)$. We prove this claim by induction on $\delta$. For $\delta=0$, this is trivial, $V(X_0)=V$ so we can take $z=\varnothing$. If $\delta$ is a limit, by the continuity of $X_\delta$, it follows that for $\gamma<\delta$, since $V(X_\gamma)=V(z_\gamma)$ for some $z_\gamma$, then $V(X_\gamma)=V(\power^{\beta+\gamma+1}(\eta'_\gamma)^M)$. Therefore, $V(X_\delta)=V(\power^{\beta+\delta}(\eta'_\delta)^M)$ as wanted.

It remains to prove this in the successor case for $\delta+1$. Note that $V(X_\delta)=V(z)$ for some $z\subseteq\power^{\beta+\delta}(\eta_\delta')^M$, therefore we can fix some $\QQ\in V(z)$ such that for some $V(z)$-generic filter, $H$, $V(z)[H]=V[G]$, by the choice of $\eta_{\delta+1}'$ there is a surjection from $\power^{\beta+\delta}(\eta_{\delta+1}')^{V(z)}$ onto $\QQ$ in $V(z)$. By \autoref{lemma:cBA-subset}, for every $x\in X_{\delta+1}$, since $x\subseteq V(z)$, there is some $y\subseteq\power^{\beta+\delta}(\eta_{\delta+1}')$ such that $V(z,x)=V(z,y)$. Therefore, \[V(X_{\delta+1})\subseteq V(\power^{\beta+\delta+1}(\eta_{\delta+1}')^M).\]
Therefore, $V(X_{\delta+1})=V(\{y\subseteq\power^{\beta+\delta}(\eta_{\delta+1}')^M\mid V(z')\subseteq V(X_{\delta+1})\})$.

Finally, we define the sequence $\tup{\eta_\alpha\mid\alpha\in\Ord}$. If $M\models\KWP_\alpha^*$ and $M=V(z)$, by \autoref{lemma:a-set-gen} there is a minimal $\eta^z$ and $y\subseteq\power^\alpha(\eta^z)$ such that $M=V(y)$. We let \[\eta_\alpha=\sup\{\eta^z\mid z\subseteq\power^{\beta+\alpha}(\eta_{\alpha+1}'), V(z)\models\KWP_\alpha^*\}.\]

Given any intermediate $M\models\KWP_\alpha^*$, we have by the first part that $M=V(z)$ for some $z\subseteq\power^{\beta+\alpha}(\eta_{\alpha+1}')$, and by the definition of $\eta_\alpha$, there is some $y\subseteq\power^\alpha(\eta_\alpha)$ such that $M=V(y)$.
\end{proof}
\section{Generalisation and applications}
\begin{theorem}
  Suppose that $V\subseteq M\subseteq V[G]$ are models of $\ZF$, where $G\subseteq\PP$ is a $V$-generic filter, then if $M\models\KWP_\alpha^*$, then $M=V(\power^{\alpha+1}(\PP^{<\omega})^M)$.
\end{theorem}
\begin{proof}
  Let $X_\delta=\power^\delta(\Ord)^M$, we claim that for all $\delta$, $V(X_\delta)=V(\power^\delta(\PP^{<\omega})^M)$. Since $M=V(X_{\alpha+1})$, by \autoref{cor:kwp-iterated-pset}, this will complete the proof. To simplify the notation let us denote by $x_\delta=\power^\delta(\PP^{<\omega})^M$. We use $\dot x_\delta\in V$ to denote the $\PP$-name for $x_\delta$, given by the construction in the proof of \autoref{thm:kwp-down-abs}.

  We prove the claim by induction on $\delta$. For $\delta=0$, this is trivial, since $X_0=\Ord$ and $x_0\in V$. If $\delta$ is a limit ordinal, this follows by the continuity of the sequences $X_\delta$ and $x_\delta$. It remains to prove the successor case.

  Let $x\in X_{\delta+1}$, then $x\subseteq V(X_\delta)=V(x_\delta)$. Let $H\subseteq x_\delta^{<\omega}$ be a $V[G]$-generic filter. We make the following observations.
  \begin{enumerate}
  \item $G\ast H$ is a $V$-generic filter for $\PP\ast\dot x_\delta^{<\omega}$.
  \item $V(x_\delta)[H]$ is a generic extension of $V(x_\delta)$ and it is an intermediate model between $V$ and $V[G\ast H]$. Thus, $V[G\ast H]$ is a generic extension of $V(x_\delta)[H]$ by a quotient of $B(\PP\ast\dot x_\delta^{<\omega})^V$, where $B(\cdot)$ is the Boolean completion of the forcing. In particular, there is a $x_\delta^{<\omega}$-name, $\dot I\in V(x_\delta)$, for an ideal defining this quotient of the Boolean algebra.
  \item $V[G\ast H]$ is a generic extension of $V(x_\delta)$ by $\QQ=x_\delta^{<\omega}\ast B(\PP\ast\dot x_\delta^{<\omega})^V/\dot I$.
  \end{enumerate}
  We define a map in $V(x_\delta)$ from $x_\delta^{<\omega}\times\PP\times\power^\delta(\PP^{<\omega})^V$ to $\QQ$, defined in $V(x_\delta)$, given by $\tup{c,p,y}\mapsto\tup{c,\tup{\check p,\tau(y)}_I}$, where $\tau(y)$ is the name for an element of $x_\delta$ decoded from $y$ via the same $\tau$ function as in the proof of \autoref{thm:kwp-down-abs}, and the notation $\tup{p,\tau(y)}_I$ denotes a name for the equivalence class of the condition in the quotient by $I$. It is easy to check that the image of this map is dense in $\QQ$.

  Moreover, it is to verify that $x_\delta\times\PP\times\power^\delta(\PP^{<\omega})^V$ is in bijection with $x_\delta$, so we get a function from $x_\delta$ onto $\QQ$ whose image is dense. By \autoref{lemma:cBA-subset} we have that $V(x_\delta,x)=V(x_\delta,y)$ for some $y\subseteq x_\delta$. Therefore, $V(x_{\delta+1})=V(X_{\delta+1})$.
\end{proof}
The extension of the Kinna--Wagner Conjecture given in \autoref{thm:KWC} gives us an interesting immediate corollary.
\begin{corollary}
  Suppose that $V\models\KWP$ and $V[G]$ is a generic extension of $V$, then there are only set-many intermediate models of $\KWP_\alpha^*$ for any given $\alpha$.\qed
\end{corollary}
As not every model of $\ZF$ satisfies $\KWP$, one has to wonder about the necessity of the $\KWP$ assumption in \autoref{thm:KWC}. It turns out that we can relativise the notion of $\KWP$ principles to any class (not necessarily a definable class, as working in $\ZF$ the problem of ground model definability remains wide open at this time).
\begin{definition}
  Suppose that $V\subseteq W$ is a class. We write $W\models\KWP_\alpha(V)$ to mean that for every $x\in W$ there is some $\alpha$ such that $x$ injects into $\power^\alpha(V)$. Namely, there is some $y\in V$ such that $x\leq\power^\alpha(y)$. We define $\KWP_\alpha^*(V)$ using surjections, and we omit the subscripts using the same convention as before.
\end{definition}
In this generalised notion, $\KWP$ is simply $\KWP(\Ord)$. It is also not hard to see that if $G$ is a $V$-generic filter, then $V[G]\models\KWP_0^*(V)$, and much of the basic properties of $\KWP$ transfer almost directly to this relativised case. Note that $V[G]\models\KWP_0(V)$ can very well be false, since that would imply that if $V\models\KWP_\alpha$, then $V[G]\models\KWP_\alpha$, but as we mentioned, Monro proved in \cite{Monro:Generic} that $\KWP_1$ can be violated generically, so $\KWP_0(V)$ must fail in that generic extension.

\begin{definition}
  Suppose that $V\subseteq W$ is a class. We say that $W\models\SVC_V(X)$ if for every set $A$ there is some $e\in V$ and a surjection $f\colon X\times e\to A$. We write $W\models\SVC_V$ to mean that there exists some $X$ such that $\SVC_V(X)$ holds.
\end{definition}
Much like in the case of $\SVC$, $W\models\SVC_V$ if and only if $\KWP_0^*(V)$ can be forced over $W$. We can also define the injective version of $\SVC_V$ as well, however, it is not at all true that the two are equivalent for an arbitrary $V$.

\begin{theorem}\label{thm:general-form}
  Suppose that $V\subseteq M\subseteq V[G]$ are models of $\ZF$ where $G$ is a $V$-generic filter. Then the following are equivalent:
  \begin{enumerate}
  \item $M=V(x)$ for some $x\in V[G]$.
  \item $M$ is a symmetric extension of $V$.
  \item $M\models\KWP^*(V)$.
  \item $M\models\SVC_V$.
  \end{enumerate}
\end{theorem}
\begin{proof}
The equivalence between (1) and (2) had been established by Grigorieff \cite{Grigorieff:Intermediate}. The implication from (1) to (4) follows directly from the definitions of $V(x)$ and $\SVC_V$. The proof that (4) implies (3) has a similar proof to \autoref{prop:SVCtoKWP}. Finally, the proof of \autoref{thm:KWC} translates directly\footnote{Mutantis mutandi.} to this case, establishing the implication from (3) to (1).
\end{proof}

Theorem~1.3 in \cite{HayutShani} argues that in their Bristol-like models, $M$, $\KWP$ and $\SVC$ are both false. As a consequence of \autoref{thm:general-form} we actually have that the two failures go hand-in-hand when $V\models\ZFC$. And indeed, $\KWP^*(V)$ fails if and only if $\SVC_V$ fails.
\subsection*{Acknowledgements}
The authors would like to thank Elliot Glazer for suggesting the concept of $\KWP^*$, as well as the anonymous referee for their helpful suggestions. The authors were supported by a UKRI Future Leaders Fellowship [MR/T021705/2]. No data are associated with this article. For the purpose of open access, the authors have applied a CC-BY licence to any Author Accepted Manuscript version arising from this submission.
\bibliographystyle{amsplain}
\providecommand{\bysame}{\leavevmode\hbox to3em{\hrulefill}\thinspace}
\providecommand{\MR}{\relax\ifhmode\unskip\space\fi MR }
\providecommand{\MRhref}[2]{%
  \href{http://www.ams.org/mathscinet-getitem?mr=#1}{#2}
}
\providecommand{\href}[2]{#2}

\end{document}